\documentclass[12pt,english]{article}
\usepackage[T1]{fontenc}
\usepackage[latin9]{luainputenc}
\usepackage{textcomp}
\usepackage{amsmath}
\usepackage{amsthm}
\usepackage{amssymb}

\makeatletter
\theoremstyle{plain}
\newtheorem{thm}{\protect\theoremname}
  \theoremstyle{plain}
  \newtheorem{conjecture}[thm]{\protect\conjecturename}
\usepackage{etoolbox}
\theoremstyle{plain}
\newtheorem{axp}{\protect\axiomname}
\makeatletter
\patchcmd{\axp}{\th@plain}{\th@plain\setcounter{axp}{\numexpr\value{thm}-1}}{}{}

\makeatother
  \theoremstyle{plain}
  \newtheorem{prop}[thm]{\protect\propositionname}
\theoremstyle{plain}
\newtheorem{notation}[thm]{\protect\notationname}
  \theoremstyle{plain}
  \newtheorem{lem}[thm]{\protect\lemmaname}
  \theoremstyle{remark}
  \newtheorem*{rem*}{\protect\remarkname}

\date{}

\makeatother

\usepackage{babel}
  \providecommand{\conjecturename}{Conjecture}
  \providecommand{\lemmaname}{Lemma}
  \providecommand{\propositionname}{Proposition}
  \providecommand{\remarkname}{Remark}
\providecommand{\axiomname}{Conjecture}
\providecommand{\notationname}{Notation}
\providecommand{\theoremname}{Theorem}

\begin{document}

\title{Harmonic functions vanishing on a cone}

\author{Dan Mangoubi and Adi Weller Weiser}
\maketitle
\begin{abstract}
Let $Z$ be a quadratic harmonic cone in $\mathbb{R}^{3}$. We consider
the family $\mathcal{H}(Z)$ of all harmonic functions vanishing on
$Z$. Is $\mathcal{H}(Z)$ finite or infinite dimensional? Some aspects
of this question go back to as early as the 19th century. To the best
of our knowledge, no nondegenerate quadratic harmonic cone exists
for which the answer to this question is known. In this paper we study
the right circular harmonic cone and give evidence that the family
of harmonic functions vanishing on it is, maybe surprisingly, finite
dimensional. We introduce an arithmetic method to handle this question
which extends ideas of Holt and Ille and is reminiscent of Hensel's
Lemma.
\end{abstract}

\section{Introduction}

\subsection{Background}

Consider the family $\mathcal{H}(Z)$ of harmonic functions in the
unit ball $B\subset\mathbb{R}^{n}$ vanishing on a given set $Z\subseteq B$.
It was conjectured in \cite{Mangoubi} and was completely proved by
Logunov and Malinnikova in \cite{malinnikova} and \cite{malinnikovaRatiosSameZeros}
that $\mathcal{H}(Z)$ possesses compactness properties. More precisely,
one can prove a Harnack type inequality for the quotient of two functions
in $\mathcal{H}(Z)$. In $\mathbb{R}^{2}$ the family $\mathcal{H}(Z)$
is locally infinite dimensional (see \cite{Mangoubi} for examples).
In higher dimensions few examples of infinite dimensional $\mathcal{H}(Z)$
are known. In fact, all known examples stem from two dimensional ones
(see \cite[§4.2]{malinnikova}). In particular, in dimension $n=3$
it is not even known whether there exists an infinite dimensional
family $\mathcal{H}(Z)$ where $Z$ is a nondegenerate quadratic harmonic
cone (it may be worth mentioning that in dimension $n=4$ there exists
such an example, see \cite[§4.2]{malinnikova}). It turns out that
this question and similar ones attracted the attention of several
mathematicians. 

Maybe the oldest closely related problem is a classical conjecture
by Stieltjes (in a letter to Hermite \cite[Letter 275]{Stieltjes}),
which concerns arithmetic properties of harmonic functions in $\mathbb{R}^{3}$
which are invariant under rotations around some axis. The present
work considers the rotationally \emph{equivariant} cases (see details
in §\ref{subsec:Results and Methods}).

Second, an analogous question was raised and solved in the context
of Bessel functions. Siegel \cite{Seigel} proved Bourget's hypothesis
that no two distinct Bessel functions have common zeros (see also
\cite[pp. 484-485]{Watson-Treaties-Bessel}). To make the resemblance
clear we note that the problem we treat here can be formulated as
whether an associated Legendre function $P_{l}^{m}$ has a common
root with the Legendre polynomial $P_{2}$ (see §\ref{sec:Equivalence-between-existence}).

Third, as a possible application to the wave equation, Agranovsky
and Krasnov raised in~\cite{agranovsky-quadraticDivisorsOfHarmonic}
the conjecture that there exists a quadratic harmonic cone $Z\subset\mathbb{R}^{3}$
such that $\mathcal{H}(Z)$ is finite dimensional.

Last, a spectral theory point of view of the same problem was given
recently by Bourgain and Rudnick in \cite{rudnick-nodaltorus}. That
work shows that given a curve of positive curvature on the standard
two dimensional flat torus there exist only a finite number of Laplace
eigenfunctions vanishing on that curve. In the case of the sphere,
an analogous question would be: Let $\gamma\subset S^{2}$ be a curve
of constant latitude which is not the equator. Do there exist only
a finite number of eigenfunctions vanishing on $\gamma$? This question
is still open, and the current work can be considered as treating
a special case of it.

The aim of the present paper is to study the family $\mathcal{H}(Z)$
where $Z$ is the right circular harmonic cone in $\mathbb{R}^{3}$.
In some sense, this is the simplest nondegenerate harmonic zero set
in $\mathbb{R}^{3}$. We give evidence that this family is finite
dimensional, while introducing a new method for handling this question.

\subsection{Results and Methods\label{subsec:Results and Methods}}

\subsubsection{Main Result}

Let us formulate the following Conjecture.
\begin{conjecture}
\label{Conjecture:finite-share-zeros-b=00003D1}Let $Z=\left\{ x^{2}+y^{2}-2z^{2}=0\right\} \cap B_{1}$
where $B_{1}\subset\mathbb{R}^{3}$ is the unit ball. Consider the
family 
\[
\mathcal{H}\left(Z\right)=\left\{ u:B_{1}\rightarrow\mathbb{R}|\Delta u=0\text{ and }u|_{Z}=0\right\} \,.
\]
Then $\mathcal{H}\left(Z\right)$ is finite dimensional.
\end{conjecture}
Using standard tools of harmonic analysis (see \cite{Armitage-conesHarmonicVanish}
and §\ref{sec:Equivalence-between-existence}) it is not difficult
to show that Conjecture \ref{Conjecture:finite-share-zeros-b=00003D1}
is equivalent to the following one, which concerns the associated
Legendre functions $P_{l}^{m}$. The Legendre polynomials ($m=0$)
will simply be denoted by $P_{l}$. 
\begin{axp}
The number of pairs $(l,m)$ such that $P_{2}|P_{l}^{m}$ is finite.\label{conj:(l,m)s.t.P2|Pl-finite}
\end{axp}
Here, for odd $m$, $P_{2}|P_{l}^{m}$ means that $P_{2}$ divides
the polynomial $\frac{P_{l}^{m}}{\sqrt{1-x^{2}}}$. The case $m=0$
of the preceding conjecture would follow from a conjecture of Stieltjes
\cite[Letter 275]{Stieltjes} concerning the irreducibility of the
Legendre polynomials over $\mathbb{Q}$. As such, it arose the interest
of several authors and, in fact, was proved by Holt \cite{Holt} and
Ille \cite{Ille}. The main new contribution of the current work comes
in the cases where $m\neq0$. We prove the following theorem (where
we include the case $m=0$ for completeness).
\begin{thm}
\label{thm:up_to_one_Plm_in_cases}For $m=0$, $P_{2}|P_{l}$ if and
only if $l=2$.

For $m=2$, $P_{2}|P_{l}^{2}$ if and only if $l=5$.

If $m$ is even, then there exists at most one $l\in\mathbb{N}$ such
that $P_{2}|P_{l}^{m}$.

If $m$ is odd, then $P_{2}\nmid P_{l}^{m}$ for all $l\in\mathbb{N}$.
\end{thm}

In fact, we prove a stronger statement in the case where $m$ is even.
We show that in some sense there exists a unique dyadic integer $l$
such that $P_{2}|P_{l}^{m}$. For the precise meaning of this please
see §\ref{subsec:Method} and Theorems \ref{theo:UniquelyExists-Hlm=00003D0mod2N+1(m=00003D0)}
and \ref{theo:UniquelyExists-Hlm=00003D0mod2N+1(m=00003D2)}. 

\subsubsection{Method\label{subsec:Method}}

The method we use in this paper consists of two steps. In the first
step, following an idea of Holt for the case $m=0$, we transform
$P_{l}^{m}$ to a polynomial $H_{l}^{m}$ whose coefficients are dyadic
integers. As such, this polynomial can be studied using modular arithmetic.
The question is whether this polynomial vanishes at the point $z=-2$.
In the second step we consider $\text{\ensuremath{H_{l}^{m}\left(-2\right)=H^{m}\left(l\right)}}$
as a (non-polynomial) function of $l$. We ask whether $H^{m}\left(l\right)\equiv0\pmod{2^{N}}$.
Our analysis shows that, if $m$ is even, $H^{m}$ is well defined
on $\mathbb{Z}/2^{N}\mathbb{Z}$ and that a solution modulus $2^{N}$
can be lifted uniquely to a solution modulus $2^{N+1}$. In this
way we get a unique dyadic integer $l$ such that $H^{m}\left(l\right)=0$
(Propositions \ref{lem:positive-first-bits-m=00003D0(4)} and \ref{prop:positive-first-bits-m=00003D2(4)};
Theorems \ref{theo:UniquelyExists-Hlm=00003D0mod2N+1(m=00003D0)}
and \ref{theo:UniquelyExists-Hlm=00003D0mod2N+1(m=00003D2)}). This
idea is reminiscent of Hensel's Lemma. However, we cannot apply Hensel's
Lemma in our case since the nature of the coefficients in Taylor's
expansion of $H^{m}$ is unclear.

\subsubsection{Secondary Results}

We describe a second approach to Conjecture \ref{Conjecture:finite-share-zeros-b=00003D1},
under significant additional assumptions. We prove
\begin{thm}
\label{Theo:pq for q harmonic, p example case}Let $f:\mathbb{R}^{3}\rightarrow\mathbb{R}$
be a harmonic function. Then the product $\left(x^{2}+y^{2}-2z^{2}\right)\cdot f\left(x,y,z\right)$
is harmonic if and only if $f\left(x,y,z\right)=\alpha+\beta xyz+\gamma\left(x^{2}-y^{2}\right)z$
for some $\alpha,\beta,\gamma\in\mathbb{R}$.
\end{thm}

Under the same assumption we can also break the rotational symmetry
of the quadratic cone and get
\begin{thm}
\label{Theorem: pq for harmonic q and p with b}Let $f:\mathbb{R}^{3}\rightarrow\mathbb{R}$
be a harmonic function and $b>1$. Then the product $\left(x^{2}+by^{2}-\left(b+1\right)z^{2}\right)\cdot f\left(x,y,z\right)$
is harmonic if and only if $f\left(x,y,z\right)=\alpha+\beta xyz$
for some $\alpha,\beta\in\mathbb{R}$.
\end{thm}

The additional assumption on the harmonicity of $f$ lets us give
a proof of Theorems \ref{Theo:pq for q harmonic, p example case}
and \ref{Theorem: pq for harmonic q and p with b} without arithmetic
considerations. Perhaps our assumptions on $f$ can be weakened.

\section*{Acknowledgments}

We are grateful to Charles Fefferman from whose discussions with the
first author the work on this topic originated. We thank Zeev Rudnick
for his interest in the present work and for helpful comments. We
thank Eran Asaf, Nir Avni, Alexander Logunov, Eugenia Malinnikova
and Amit Ophir for interesting discussions. This paper is part of
the second author\textquoteright s research towards a Ph.D.\ dissertation,
conducted at the Hebrew University of Jerusalem. The cases of $m=2$
and odd $m$ in Theorem \ref{thm:up_to_one_Plm_in_cases} together
with Theorems \ref{Theo:pq for q harmonic, p example case} and~\ref{Theorem: pq for harmonic q and p with b}
were proved in \cite{adisMasters}. We gratefully acknowledge the
support of ISF grant no.\ 753/14.\newpage{}

\section{Step 1: Holt-Ille Transformation\label{subsec:Holt-Ille-transformation}}

In this section we transform the associated Legendre function $P_{l}^{m}$
to a polynomial $H_{l}^{m}$ of degree $\lfloor\frac{l-m}{2}\rfloor$.
The main property of $H_{l}^{m}$ is that its coefficients are dyadic
integers, making it useful in analyzing whether $P_{2}$ and $P_{l}^{m}$
share a common root. This idea was developed by Holt \cite{Holt}
and Ille \cite{Ille} for the case $m=0$ and we extend it here to
the cases where $m\neq0$. 

We recall the following integral representation of the associated
Legendre functions:
\begin{equation}
P_{l}^{m}\left(x\right)=i^{m}\frac{\left(l+m\right)!}{l!}\frac{1}{\pi}\int_{0}^{\pi}\left(x+y\cos\varphi\right)^{l}\cos\left(m\varphi\right)d\varphi\label{eq:plm-integral}
\end{equation}

for $x\in\left[-1,1\right]$, where $y=i\sqrt{1-x^{2}}$ \cite[Ch.VII, p.\ 505]{courant-intrep}.
By Lemma \ref{Lemma-integral-cosine} the integral on the right hand
side is of the form $x^{\delta}y^{m}Q\left(x^{2},y^{2}\right)$ where
$\delta\in\left\{ 0,1\right\} $ with $\delta\equiv\left(l-m\right)\pmod{2}$
and $Q$ is a real homogeneous polynomial of degree $\left\lfloor \frac{l-m}{2}\right\rfloor $.
If we define 
\[
z=\frac{4x^{2}}{x^{2}-1}\:,
\]
then we can express $x^{2}$ and $y^{2}$ by

\begin{equation}
x^{2}=\frac{z}{z-4},\,\,\,\,y^{2}=\frac{4}{z-4}\,.\label{eq:x,y_by_z}
\end{equation}
Substituting the preceding expressions in $Q$ gives 
\begin{equation}
P_{l}^{m}\left(x\right)=x^{\delta}y^{m}\frac{C_{l}^{m}}{\left(z-4\right)^{\left\lfloor \frac{l-m}{2}\right\rfloor }}H_{l}^{m}\left(z\right)\label{eq:P-from-H}
\end{equation}
where $H_{l}^{m}\left(z\right)$ is a polynomial of degree $\left\lfloor \frac{l-m}{2}\right\rfloor $
normalized so that $H_{l}^{m}\left(0\right)=1$ and the constant $C_{l}^{m}$
depends on $l$ and $m$ (for details see proof of Lemma \ref{lemma:forms of Hlm}
in §\ref{sec:Technical-Lemmas}). We note
\begin{prop}
$P_{2}|P_{l}^{m}$ if and only if $H_{l}^{m}\left(-2\right)=0$.\label{prop:P2|Plm-iff-Hlm(-2)=00003D0}
\end{prop}
\begin{proof}
$P_{2}\left(x\right)=\frac{1}{2}\left(3x^{2}-1\right)$. Hence $P_{2}|P_{l}^{m}$
if and only if $P_{l}^{m}\left(\frac{1}{\sqrt{3}}\right)=0$. The
proposition now follows from (\ref{eq:x,y_by_z}) and (\ref{eq:P-from-H}).
\end{proof}
\begin{notation}
We denote by $\sigma_{l}^{m}\left(k\right)$ coefficients such that
the following holds\label{not:sigma-l-m}
\[
H_{l}^{m}\left(z\right)=\sum_{k=0}^{\left\lfloor \frac{l-m}{2}\right\rfloor }\left(-\frac{z}{2}\right)^{k}\sigma_{l}^{m}\left(k\right)
\]
\end{notation}
The formulas for $\sigma_{l}^{m}\left(k\right)$ are recorded in the
following lemma.
\begin{lem}
\label{lemma:forms of Hlm}If $l+m$ is even, then
\begin{equation}
\sigma_{l}^{m}\left(k\right)=\frac{\left(-2\right)^{k}{\left\lfloor \frac{l-m}{2}\right\rfloor  \choose k}{\left\lfloor \frac{l+m}{2}\right\rfloor  \choose k}}{{2k \choose k}}\label{eq:sigma_l+m_even}
\end{equation}
 and if $l+m$ is odd, then
\begin{equation}
\sigma_{l}^{m}\left(k\right)=\frac{\left(-2\right)^{k}{\left\lfloor \frac{l-m}{2}\right\rfloor  \choose k}{\left\lfloor \frac{l+m}{2}\right\rfloor  \choose k}}{{2k \choose k}\left(2k+1\right)}\,.\label{eq:sigma_l+m_odd}
\end{equation}

In both cases 
\begin{equation}
\sigma_{l}^{m}\left(0\right)=1\,.\label{eq:sig(0)=00003D1}
\end{equation}
\end{lem}
A proof for these formulas is included in §\ref{sec:Technical-Lemmas}.

\section{Proof of Theorem \ref{thm:up_to_one_Plm_in_cases}: The case of odd
$m$\label{sec:odd-m}}

It will be convenient to use the following

\paragraph{Notation. }

We denote 
\begin{equation}
s=\left\lfloor \frac{l}{2}\right\rfloor ,\,t=\left\lfloor \frac{m}{2}\right\rfloor \,.\label{eq:s,t=00003Dl,m/2}
\end{equation}

We first assume that $l$ is odd. We rewrite formula (\ref{eq:sigma_l+m_even})
as follows,
\[
\sigma_{2s+1}^{2t+1}\left(k\right)=\left(-1\right)^{k}{s-t \choose k}\frac{\left(s+t+1\right)\left(s+t\right)...\left(s+t-k+2\right)}{1\cdot3\cdot...\cdot\left(2k-1\right)}\,.
\]

The term $\sigma_{l}^{m}\left(1\right)$ is even, since either $\left(s-t\right)$
or $\left(s+t+1\right)$ is even. 

For $k\geq2$ we have that $\sigma_{l}^{m}\left(k\right)$ is even,
since there are at least two consecutive numbers at the nominator
of the second factor, and no even numbers in the denominator. 

Combined with (\ref{eq:sig(0)=00003D1}), it follows that $H_{l}^{m}\left(-2\right)=\sum_{k}\sigma_{l}^{m}\left(k\right)\equiv1\,\pmod{2}$.
From Proposition \ref{prop:P2|Plm-iff-Hlm(-2)=00003D0} we get that
$\mbox{\ensuremath{P_{2}\nmid P_{l}^{m}}}$.

If $l$ is even, essentially the same argument holds replacing formula
(\ref{eq:sigma_l+m_even}) by formula (\ref{eq:sigma_l+m_odd}). We
leave the details to the reader.

\section{Proof of Theorem \ref{thm:up_to_one_Plm_in_cases}: The case of $\mbox{\ensuremath{m\equiv0\pmod{4}}}$\label{sec:Proof-of-m=00003D0}}

\subsection{Recovering the lowest three bits of $l$}
\begin{prop}
If $H_{l}^{m}\left(-2\right)\equiv0\pmod{8}$ and $m\equiv0\pmod{4}$
then $l\equiv2\pmod{8}$.\label{prop:m=00003D0(4)->l=00003D2(8)}
\end{prop}
The proof is by ruling out the other possibilities one by one.
\begin{lem}
If $m\equiv0\pmod{4}$ and either $l\equiv0\pmod{4}$ or $l\equiv1\pmod{4}$
then $H_{l}^{m}\left(-2\right)\equiv1\pmod{2}$.\label{lem:m=00003D0(4),l=00003D0(4)}
\end{lem}
\begin{proof}
In these cases both $s$ and $t$ (see (\ref{eq:s,t=00003Dl,m/2}))
are even. Rewriting formulas (\ref{eq:sigma_l+m_even}) and (\ref{eq:sigma_l+m_odd})
modulus 2, we get
\[
\sigma_{l}^{m}\left(k\right)\equiv{s-t \choose k}\left(s+t\right)\left(s+t-1\right)...\left(s+t-k+1\right)\pmod{2}\,.
\]
For every $k\geq1$ we have that $\sigma_{l}^{m}\left(k\right)$ is
even, since it has the even factor $\left(s+t\right)$. Summing over
$k\geq0$ we get $H_{l}^{m}\left(-2\right)\equiv1\pmod{2}$.
\end{proof}
\begin{lem}
If $m\equiv0\pmod{4}$ and $l\equiv3\pmod{4}$ then $H_{l}^{m}\left(-2\right)\equiv2\pmod{4}$.
\end{lem}
\begin{proof}
Here $s$ is odd and $t$ is even. We calculate $H_{2s+1}^{2t}\left(-2\right)\pmod{4}$
using formula (\ref{eq:sigma_l+m_odd}). 
\[
\sigma_{2s+1}^{2t}\left(1\right)=-\frac{\left(s-t\right)\left(s+t\right)}{3}=-\frac{s^{2}-t^{2}}{3}\equiv-\frac{1-0}{-1}\equiv1\pmod{4}
\]

and
\[
\sigma_{2s+1}^{2t}\left(3\right)=-\sigma_{2s+1}^{2t}\left(2\right)\frac{\left(s-t-2\right)\left(s+t-2\right)}{3\cdot7}\equiv-\sigma_{2s+1}^{2t}\left(2\right)\pmod{4}\,.
\]

From basic divisibility properties (Lemma \ref{lem:power2 bigger than !})
we also have

\[
\forall k\geq4,\:\sigma_{l}^{m}\left(k\right)\equiv0\pmod{4}\:.
\]

Summing over $k\geq0$ we get $H_{2s+1}^{2t}\left(-2\right)\equiv2\pmod{4}$.
\end{proof}
\begin{lem}
If $m\equiv0\pmod{4}$ and $l\equiv6\pmod{8}$ then $H_{l}^{m}\left(-2\right)\equiv4\pmod{8}$.\label{lem:l=00003D2mod8forM=00003D0mod4}
\end{lem}
\begin{proof}
We let $l=8q+6$. By formula (\ref{eq:sigma_l+m_even})
\[
\sigma_{l}^{m}\left(1\right)=-\left(4q+3-t\right)\left(4q+3+t\right)\equiv t^{2}-1\pmod{8}\,,
\]
\[
\sigma_{l}^{m}\left(2\right)=-\sigma_{l}^{m}\left(1\right)\frac{\left(4q+2-t\right)\left(4q+2+t\right)}{2\cdot3}\equiv-\sigma_{l}^{m}\left(1\right)\frac{4-t^{2}}{2\cdot3}\pmod{8}
\]

and
\[
\sigma_{l}^{m}\left(3\right)=-\sigma_{l}^{m}\left(2\right)\frac{\left(4q+1\right)^{2}-t^{2}}{3\cdot5}\equiv\sigma_{l}^{m}\left(2\right)\pmod{8}\,,
\]
 where in the last calculation we used the fact that $\sigma_{l}^{m}\left(2\right)\equiv0\pmod{2}$
(Lemma \ref{lem:power2 bigger than !}). From basic divisibility properties
(Lemma \ref{lem:power2 bigger than !}) we have $\sigma_{l}^{m}\left(k\right)\equiv0\pmod{8}$
for all $k\geq4$. Summing over $k\geq0$ we get $H_{l}^{m}\left(-2\right)\equiv4\pmod{8}$.

\end{proof}
To be complete we verify the remaining case $l\equiv2\pmod{8}$.
\begin{prop}
If $m\equiv0\pmod{4}$ and $l\equiv2\pmod{8}$ then $H_{l}^{m}\left(-2\right)\equiv0\pmod{8}$.\label{lem:positive-first-bits-m=00003D0(4)}
\end{prop}
\begin{proof}
We let $l=8q+2$. By formula (\ref{eq:sigma_l+m_even})
\[
\sigma_{l}^{m}\left(1\right)=-\left(4q+1-t\right)\left(4q+1+t\right)\equiv t^{2}-1\pmod{8}\,,
\]
\[
\sigma_{l}^{m}\left(2\right)=-\sigma_{l}^{m}\left(1\right)\frac{\left(4q-t\right)\left(4q+t\right)}{2\cdot3}\equiv\sigma_{l}^{m}\left(1\right)\frac{t^{2}}{2\cdot3}\pmod{8}
\]

and
\[
\sigma_{l}^{m}\left(3\right)=-\sigma_{l}^{m}\left(2\right)\frac{\left(4q-1\right)^{2}-t^{2}}{3\cdot5}\equiv\sigma_{l}^{m}\left(2\right)\pmod{8}\,,
\]
From basic divisibility properties (Lemma \ref{lem:power2 bigger than !})
we have $\sigma_{l}^{m}\left(k\right)\equiv0\pmod{8}$ for all $k\geq4$.
Summing over $k\geq0$ we get $H_{l}^{m}\left(-2\right)\equiv0\pmod{8}$.
\end{proof}

\subsection{Recovering the high bits of $l$}

In this section we introduce an idea in the spirit of Hensel's lemma
to recover the high bits of $l$. 
\begin{thm}
\label{theo:UniquelyExists-Hlm=00003D0mod2N+1(m=00003D0)}Let $m\equiv0\pmod{4}$
and suppose that $H_{l}^{m}\left(-2\right)\equiv0\pmod{2^{N}}$. Then
there exists a unique $l\leq\tilde{l}<l+2^{N+1}$ such that $H_{\tilde{l}}^{m}\left(-2\right)\equiv0\pmod{2^{N+1}}$.
Moreover, $\tilde{l}\equiv l\pmod{2^{N}}$. 
\end{thm}
Using Proposition \ref{prop:m=00003D0(4)->l=00003D2(8)}, Theorem
\ref{theo:UniquelyExists-Hlm=00003D0mod2N+1(m=00003D0)} is an immediate
corollary of

\begin{prop}
Let $m\equiv0\pmod{4}$ and $l\equiv2\pmod{8}$. Fix $N\geq3$ and
write $l=r+2^{N}q$ with $0\leq r<2^{N}$. Then $H_{l}^{m}\left(-2\right)\equiv H_{r}^{m}\left(-2\right)+2^{N}q\pmod{2^{N+1}}$.\label{prop:epsilon-coefficent-m=00003D0(4)}
\end{prop}
\begin{rem*}
In particular, it follows that if $l\equiv\tilde{l}\pmod{2^{N}}$
then $H_{l}^{m}\left(-2\right)\equiv H_{\tilde{l}}^{m}\left(-2\right)\pmod{2^{N}}$. 
\end{rem*}
To prove Proposition \ref{prop:epsilon-coefficent-m=00003D0(4)} we
first observe
\begin{lem}
Let $m\equiv0\pmod{4}$, $l\equiv0\pmod{2}$ and $k\geq6$. If $\tilde{l}\equiv l\pmod{2^{N}}$
then $\sigma_{\tilde{l}}^{m}\left(k\right)\equiv\sigma_{l}^{m}\left(k\right)\pmod{2^{N+1}}$.\label{lemma:epsilon-coef-0-for-k>6}
\end{lem}
We postpone the proof of this Lemma to the end of the section.
\begin{proof}[Proof of Proposition \ref{prop:epsilon-coefficent-m=00003D0(4)}]

By Lemma \ref{lemma:epsilon-coef-0-for-k>6}, for all $k\geq6$ we
have that $\sigma_{l}^{m}\left(k\right)\equiv\sigma_{r}^{m}\left(k\right)\pmod{2^{N+1}}$.
For $0\leq k\leq5$ we calculate $\sigma_{l}^{m}\left(k\right)$ explicitly. 

From (\ref{eq:sig(0)=00003D1}) we have $\sigma_{l}^{m}\left(0\right)=\sigma_{r}^{m}\left(0\right)=1$.

Let $\tilde{s}=\frac{r}{2}$. From the assumptions we have $\tilde{s}\equiv1\pmod{4}$.
Using this and the assumptions that $N\geq3$ and $t$ is even we
get the following expressions for the next terms.
\[
\sigma_{l}^{m}\left(1\right)=-\left(2^{N-1}q+\tilde{s}-t\right)\left(2^{N-1}q+\tilde{s}+t\right)\equiv2^{N}q+\sigma_{r}^{m}\left(1\right)\pmod{2^{N+1}}\,.
\]
\[
\sigma_{l}^{m}\left(2\right)=-\sigma_{l}^{m}\left(1\right)\frac{\left(2^{N-1}q+\tilde{s}-t-1\right)\left(2^{N-1}q+\tilde{s}+t-1\right)}{2\cdot3}\,.
\]
Elementary manipulations of this expression, noticing that $\tilde{s}\equiv1\pmod{4}$,
$N\geq3$, $\tilde{s}+t-1$ is even and that $\sigma_{r}^{m}\left(1\right)$
is odd, give
\[
\sigma_{l}^{m}\left(2\right)\equiv2^{2N-3}q+\sigma_{r}^{m}\left(2\right)\pmod{2^{N+1}}\,.
\]
Moving on to the next term, using again that $N\geq3$ and noticing
that $\sigma_{l}^{m}\left(2\right)$ is even we get
\[
\sigma_{l}^{m}\left(3\right)=-\sigma_{l}^{m}\left(2\right)\frac{\left(2^{N-1}q+\tilde{s}-t-2\right)\left(2^{N-1}q+\tilde{s}+t-2\right)}{3\cdot5}
\]
\[
\equiv2^{2N-3}q+\sigma_{r}^{m}\left(3\right)\pmod{2^{N+1}}\,.
\]

At this point we observe that $\sigma_{l}^{m}\left(2\right)+\sigma_{l}^{m}\left(3\right)\equiv\sigma_{r}^{m}\left(2\right)+\sigma_{r}^{m}\left(3\right)\pmod{2^{N+1}}$
(since $N\geq3$). A similar circumstance occurs in the next two terms.
\[
\sigma_{l}^{m}\left(4\right)=-\sigma_{l}^{m}\left(3\right)\frac{\left(2^{N-1}q+\tilde{s}-t-3\right)\left(2^{N-1}q+\tilde{s}+t-3\right)}{4\cdot7}
\]
\[
\equiv2^{2N-5}q\left(4-t^{2}\right)+\sigma_{r}^{m}\left(3\right)\left[2^{2N-4}q+2^{N-1}q\right]+\sigma_{r}^{m}\left(4\right)
\]
 
\[
\equiv2^{2N-3}q+2^{N-1}tq+\sigma_{r}^{m}\left(4\right)\pmod{2^{N+1}}
\]
and
\[
\sigma_{l}^{m}\left(5\right)\equiv-\sigma_{l}^{m}\left(4\right)\frac{\left(2^{N-1}q+\tilde{s}-t-4\right)\left(2^{N-1}q+\tilde{s}+t-4\right)}{5\cdot9}
\]
\[
\equiv2^{2N-3}q+2^{N-1}tq+\sigma_{r}^{m}\left(5\right)\pmod{2^{N+1}}
\]

Summing over $k\geq0$ we get that $H_{l}^{m}\left(-2\right)\equiv2^{N}q+H_{r}^{m}\left(-2\right)\pmod{2^{N+1}}$.
\end{proof}
It remains to prove Lemma \ref{lemma:epsilon-coef-0-for-k>6}.
\begin{proof}[Proof of Lemma \ref{lemma:epsilon-coef-0-for-k>6}]

For any $x\in\mathbb{Q}$ let $v_{2}\left(x\right)$ be the dyadic
valuation of $x$ (see Notation~\ref{notation:The-dyadic-valuation.}).
For $k$ such that $v_{2}\left(k!\right)\geq N+1$ we have that $\sigma_{l}^{m}\left(k\right)\equiv0\pmod{2^{N+1}}$
for all $l$ (see Lemma \ref{lem:power2 bigger than !}). Hence we
may assume that 
\begin{equation}
v_{2}\left(k!\right)\leq N\,.\label{eq:can_assume_v2(k!)<=00003DN}
\end{equation}

In particular, since $k\geq6$ it follows that 
\begin{equation}
N\geq4\,.\label{eq:Nbigger3Ink>6}
\end{equation}

Let $s=\lfloor\frac{l}{2}\rfloor=2^{N-1}q+\tilde{s}$ with $0\leq\tilde{s}<2^{N-1}$.
Collecting terms in (\ref{eq:sigma_l+m_even}) according to the powers
of $q$ we see that 
\[
\sigma_{l}^{m}\left(k\right)=\sum_{i=0}^{k}\sum_{j=0}^{k}\underbrace{\frac{1}{k!\left(2k-1\right)!!}2^{\left(N-1\right)\left(i+j\right)}P_{k-i,k}\left(\tilde{s}-t\right)P_{k-j,k}\left(\tilde{s}+t\right)}_{a_{i,j,k,l,m}}q^{i+j}
\]
where $P_{i,k}\left(x\right)$ is the elementary symmetric polynomial
of degree $i$ in the $k$ variables $\left\{ x,x-1,...,x-k+1\right\} $.

We now show that all the coefficients $a_{i,j,k,l,m}$ with $i+j>0$
vanish modulus~$2^{N+1}$. This is enough since $\tilde{s}$ is determined
by $l$ modulus $2^{N}$.

Case (i): $i+j\geq4$. We use (\ref{eq:can_assume_v2(k!)<=00003DN})
and (\ref{eq:Nbigger3Ink>6}) to get $v_{2}\left(a_{i,j,k,l,m}\right)\geq4\left(N-1\right)-v_{2}\left(k!\right)\geq3N-4>N+1$.

Case (ii): $i+j=1$. We may assume $i=0,j=1$. Here $v_{2}\left(P_{k-i,k}\left(\tilde{s}-t\right)\right)=v_{2}\left(P_{k,k}\left(\tilde{s}-t\right)\right)\geq v_{2}\left(k!\right)$
(see proof of Lemma \ref{lem:power2 bigger than !}). Since $k\geq6$
we also have $v_{2}\left(P_{k-j,k}\left(\tilde{s}+t\right)\right)=v_{2}\left(P_{k-1,k}\left(\tilde{s}+t\right)\right)\geq2$.
So $v_{2}\left(a_{i,j,k,l,m}\right)\geq N-1+v_{2}\left(k!\right)+2-v_{2}\left(k!\right)\geq N+1$. 

Case (iii): $3\geq i+j\geq2$. Here a direct examination of the few
possibilities, taking into account that $k\geq6$, shows that $v_{2}\left(P_{k-i,k}\left(\tilde{s}-t\right)\right)+v_{2}\left(P_{k-j,k}\left(\tilde{s}+t\right)\right)\geq3$.
We use (\ref{eq:can_assume_v2(k!)<=00003DN}) and get $v_{2}\left(a_{i,j,k,l,m}\right)\geq2\left(N-1\right)+3-v_{2}\left(k!\right)\geq2N+1-N=N+1$.
\end{proof}

\section{Proof of Theorem \ref{thm:up_to_one_Plm_in_cases}: The case of $\mbox{\ensuremath{m\equiv2\pmod{4}}}$}

The arguments in this section are similar to the ones in §\ref{sec:Proof-of-m=00003D0}.

\subsection{Recovering the lowest three bits of $l$}
\begin{prop}
If $H_{l}^{m}\left(-2\right)\equiv0\pmod{8}$ and $m\equiv2\pmod{4}$
then $l\equiv5\pmod{8}$. \label{prop:m=00003D2(4)->l=00003D5(8)}
\end{prop}
The proof is by ruling out the other possibilities one by one.
\begin{lem}
If $m\equiv2\pmod{4}$ and either $l\equiv2\pmod{4}$ or $l\equiv3\pmod{4}$
then $H_{l}^{m}\left(-2\right)\equiv1\pmod{2}$.
\end{lem}
\begin{proof}
In these cases both $s$ and $t$ are odd. Repeating the proof of
Lemma~\ref{lem:m=00003D0(4),l=00003D0(4)}, rewriting formulas (\ref{eq:sigma_l+m_even})
and (\ref{eq:sigma_l+m_odd}) modulus~2, we get
\[
\sigma_{l}^{m}\left(k\right)\equiv{s-t \choose k}\left(s+t\right)\left(s+t-1\right)...\left(s+t-k+1\right)\pmod{2}\,.
\]
For every $k\geq1$ we have that $\sigma_{l}^{m}\left(k\right)$ is
even, since it has the even factor~$\left(s+t\right)$. So $H_{l}^{m}\left(-2\right)\equiv1\pmod{2}$. 
\end{proof}

\begin{lem}
If $m\equiv2\pmod{4}$ and $l\equiv0\pmod{4}$ then $H_{l}^{m}\left(-2\right)\equiv2\pmod{4}$.
\end{lem}
\begin{proof}
Here $s$ is even and $t$ is odd. We calculate $H_{2s}^{2t}\left(-2\right)\pmod{4}$
using formula (\ref{eq:sigma_l+m_even}).
\[
\sigma_{2s}^{2t}\left(1\right)=-\left(s-t\right)\left(s+t\right)=t^{2}-s^{2}\equiv1\pmod{4}
\]
and
\[
\sigma_{2s}^{2t}\left(3\right)=-\sigma_{2s}^{2t}\left(2\right)\frac{\left(s-t-2\right)\left(s+t-2\right)}{3\cdot5}\equiv-\sigma_{2s}^{2t}\left(2\right)\pmod{4}\:.
\]
From basic divisibility properties (Lemma \ref{lem:power2 bigger than !})
we also have
\[
\forall k\geq4,\:\sigma_{l}^{m}\left(k\right)\equiv0\pmod{4}\:.
\]
Summing over $k\geq0$ we get $H_{2s}^{2t}\left(-2\right)\equiv2\pmod{4}$.
\end{proof}
\begin{lem}
If $m\equiv2\pmod{4}$ and $l\equiv1\pmod{8}$ then $H_{l}^{m}\left(-2\right)\equiv4\pmod{8}$.
\end{lem}
\begin{proof}
We denote $l=8q+1$ and look at the terms of $H_{l}^{m}\left(-2\right)\pmod{8}$
individually using formula (\ref{eq:sigma_l+m_odd}). Using the assumption
that $t$ is odd we get
\[
\sigma_{l}^{m}\left(1\right)=-\frac{\left(4q-t\right)\left(4q+t\right)}{3}\equiv3\pmod{8}\:,
\]
and by Lemma \ref{lem:v2(sig(2))>2} 
\[
\sigma_{l}^{m}\left(3\right)=-\sigma_{l}^{m}\left(2\right)\frac{\left(4q-2\right)^{2}-t^{2}}{3\cdot7}\equiv-\sigma_{l}^{m}\left(2\right)\pmod{8}\,.
\]

From basic divisibility properties (Lemma \ref{lem:power2 bigger than !})
we have $\sigma_{l}^{m}\left(k\right)\equiv0\pmod{8}$ for all $k\geq4$.
Summing over $k\geq0$ we get $H_{l}^{m}\left(-2\right)\equiv4\pmod{8}$.
\end{proof}
To be complete we verify the remaining case for $l\equiv5\pmod{8}$.
\begin{prop}
If $m\equiv2\pmod{4}$ and $l\equiv5\pmod{8}$ then $H_{l}^{m}\left(-2\right)\equiv0\pmod{8}$.\label{prop:positive-first-bits-m=00003D2(4)}
\end{prop}
\begin{proof}
We let $l=8q+5$. By formula (\ref{eq:sigma_l+m_odd})
\[
\sigma_{l}^{m}\left(1\right)=-\frac{\left(4q+2-t\right)\left(4q+2+t\right)}{3}\equiv-1\pmod{8}\:,
\]
and by Lemma \ref{lem:v2(sig(2))>2}
\[
\sigma_{l}^{m}\left(3\right)=-\sigma_{l}^{m}\left(2\right)\frac{\left(4q\right)^{2}-t^{2}}{3\cdot7}\equiv-\sigma_{l}^{m}\left(2\right)\pmod{8}\,.
\]

From basic divisibility properties (Lemma \ref{lem:power2 bigger than !})
we have $\sigma_{l}^{m}\left(k\right)\equiv0\pmod{8}$ for all $k\geq4$.
Summing over $k\geq0$ we get $H_{l}^{m}\left(-2\right)\equiv0\pmod{8}$.
\end{proof}

\subsection{Recovering the high bits of $l$}
\begin{thm}
\label{theo:UniquelyExists-Hlm=00003D0mod2N+1(m=00003D2)}Let $m\equiv2\pmod{4}$
and suppose that $H_{l}^{m}\left(-2\right)\equiv0\pmod{2^{N}}$. Then
there exists a unique $l\leq\tilde{l}<l+2^{N+1}$ such that $H_{\tilde{l}}^{m}\left(-2\right)\equiv0\pmod{2^{N+1}}$.
Moreover $\tilde{l}\equiv l\pmod{2^{N}}$. 
\end{thm}
Using Proposition \ref{prop:m=00003D2(4)->l=00003D5(8)}, Theorem
\ref{theo:UniquelyExists-Hlm=00003D0mod2N+1(m=00003D2)} is an immediate
corollary of
\begin{prop}
Let $m\equiv2\pmod{4}$ and $l\equiv5\pmod{8}$. Fix $N\geq3$ and
write $\mbox{\ensuremath{l=r+2^{N}q}}$ with $0\leq r<2^{N}$. Then
$H_{l}^{m}\left(-2\right)\equiv H_{r}^{m}\left(-2\right)+2^{N}q\pmod{2^{N+1}}$.\label{prop:epsilon-coefficent-m=00003D2(4)-1}
\end{prop}
To prove Proposition \ref{prop:epsilon-coefficent-m=00003D2(4)-1}
we first observe
\begin{lem}
Let $m\equiv2\pmod{4}$, $l\equiv1\pmod{2}$ and $k\geq6$. If $\tilde{l}\equiv l\pmod{2^{N}}$
then $\sigma_{\tilde{l}}^{m}\left(k\right)\equiv\sigma_{l}^{m}\left(k\right)\pmod{2^{N+1}}$.\label{prop:epsilon-coef-0-for-k>6-1}
\end{lem}
\begin{proof}
The only difference from the proof of Lemma \ref{lemma:epsilon-coef-0-for-k>6}
is a division by an odd number which does not influence the calculations.
\end{proof}
Now we move on to the proof of the main proposition of this section.
\begin{proof}[Proof of Proposition \ref{prop:epsilon-coefficent-m=00003D2(4)-1}]

By Lemma \ref{prop:epsilon-coef-0-for-k>6-1}, for all $k\geq6$ we
have that $\sigma_{l}^{m}\left(k\right)\equiv\sigma_{r}^{m}\left(k\right)\pmod{2^{N+1}}$.
For $0\leq k\leq5$ we calculate $\sigma_{l}^{m}\left(k\right)$ explicitly. 

From (\ref{eq:sig(0)=00003D1}) we have $\sigma_{l}^{m}\left(0\right)=\sigma_{r}^{m}\left(0\right)=1$.

Denote $\tilde{s}=\left\lfloor \frac{r}{2}\right\rfloor $, from the
assumptions we get that $\tilde{s}\equiv2\pmod{4}$. Using this and
the assumption $N\geq3$ we get
\[
\sigma_{l}^{m}\left(1\right)=-\frac{1}{3}\left(2^{N-1}q+\tilde{s}-t\right)\left(2^{N-1}q+\tilde{s}+t\right)\equiv\sigma_{r}^{m}\left(1\right)\pmod{2^{N+1}}\,.
\]
The next term is
\[
\sigma_{l}^{m}\left(2\right)=-\sigma_{l}^{m}\left(1\right)\frac{\left(2^{N-1}q+\tilde{s}-t-1\right)\left(2^{N-1}q+\tilde{s}+t-1\right)}{2\cdot5}\,.
\]
Elementary manipulations of this expression, noticing that $N\geq3$,
$\tilde{s}-1\equiv1\pmod{4}$ and that $\sigma_{r}^{m}\left(1\right)\equiv-1\pmod{4}$
(since $t$ is odd), give
\[
\sigma_{l}^{m}\left(2\right)\equiv2^{2N-3}q+2^{N-1}q+\sigma_{r}^{m}\left(2\right)\pmod{2^{N+1}}\,.
\]
Moving on to the next term, using again $N\geq3$ and $\tilde{s}\equiv2\pmod{4}$
we get
\[
\sigma_{l}^{m}\left(3\right)=-\sigma_{l}^{m}\left(2\right)\frac{\left(2^{N-1}q+\tilde{s}-t-2\right)\left(2^{N-1}q+\tilde{s}+t-2\right)}{3\cdot7}
\]
\[
\equiv2^{2N-3}q+2^{N-1}q+\sigma_{r}^{m}\left(3\right)\pmod{2^{N+1}}\,.
\]

At this point we observe that $\sigma_{l}^{m}\left(2\right)+\sigma_{l}^{m}\left(3\right)\equiv2^{N}q+\sigma_{r}^{m}\left(2\right)+\sigma_{r}^{m}\left(3\right)\pmod{2^{N+1}}$
(since $N\geq3$). A similar circumstance occurs in the next two terms. 

By Lemma \ref{lem:v2(sig(2))>2} $\sigma_{r}^{m}\left(2\right)\equiv0\pmod{4}$
so also $\mbox{\ensuremath{\sigma_{r}^{m}\left(3\right)\equiv0\pmod{4}}}$.
Hence, 
\[
\sigma_{l}^{m}\left(4\right)=-\sigma_{l}^{m}\left(3\right)\frac{\left(2^{N-1}q+\tilde{s}-t-3\right)\left(2^{N-1}q+\tilde{s}+t-3\right)}{4\cdot9}
\]
\[
\equiv2^{2N-3}q+2^{N-3}q\left(1-t^{2}\right)+2^{N-2}q\sigma_{r}^{m}\left(3\right)+\sigma_{r}^{m}\left(4\right)
\]
\[
\equiv2^{2N-3}q+\sigma_{r}^{m}\left(4\right)\pmod{2^{N+1}}
\]

and
\[
\sigma_{l}^{m}\left(5\right)\equiv-\sigma_{l}^{m}\left(4\right)\frac{\left(2^{N-1}q+\tilde{s}-t-4\right)\left(2^{N-1}q+\tilde{s}+t-4\right)}{5\cdot11}
\]
\[
\equiv2^{2N-3}q+\sigma_{r}^{m}\left(5\right)\pmod{2^{N+1}}
\]

Summing over $k\geq0$ we get that $H_{l}^{m}\left(-2\right)\equiv2^{N}q+H_{r}^{m}\left(-2\right)\pmod{2^{N+1}}$.
\end{proof}

\section{Proof of Theorem \ref{thm:up_to_one_Plm_in_cases}: The special cases
$m=0$ and $m=2$\label{subsec:A-note-on-m=00003D0,2,4}}

In these cases we trivially see that $P_{2}|P_{2}$ and we can easily
check that $P_{2}|P_{5}^{2}$. The general statement for even $m$
shows that these are the only solutions. Note that $P_{5}^{2}$ corresponds
to the harmonic polynomials $\left(x^{2}+y^{2}-2z^{2}\right)xyz$
and $\left(x^{2}+y^{2}-2z^{2}\right)\left(x^{2}-y^{2}\right)z$.

\section{Harmonic products of two harmonic polynomials\label{sec:Harmonic-products-of}}

We prove Theorem \ref{Theo:pq for q harmonic, p example case}, giving
evidence to the validity of Conjecture \ref{Conjecture:finite-share-zeros-b=00003D1}.
\begin{proof}[Proof of Theorem \ref{Theo:pq for q harmonic, p example case}]
We let $p\left(x,y,z\right)=x^{2}+y^{2}-2z^{2}$ and let $h=p\cdot f$.
We assume that both $h$ and $f$ are harmonic. Harmonic functions
are analytic, so they can be represented as infinite sums of homogeneous
polynomials. If $f$ is harmonic and $\Delta\left(pf\right)=0$ then
every homogeneous component of $f$, denoted $f_{d}$, is also harmonic
and has to give $\Delta\left(pf_{d}\right)=0$. So we can assume $f$
is a homogeneous harmonic polynomial of degree $d$. 

Let $f\left(x,y,z\right)=\sum_{i+j+k=d}a_{ijk}x^{i}y^{j}z^{k}$. The
Laplacian of $h$ is
\[
\Delta h=2\sum_{\alpha\in\left\{ x,y,z\right\} }\partial_{\alpha}p\partial_{\alpha}f=4\sum_{i+j+k=d}\left(i+j-2k\right)a_{ijk}x^{i}y^{j}z^{k}\,.
\]
Since the product $h$ is harmonic every coefficient of every monomial
of the above expression has to vanish. If $a_{ijk}\neq0$ then $i+j=2k$.
Combining this with $i+j+k=d$ we get $3k=d$. Hence $f$ can only
be of the form $f\left(x,y,z\right)=\sum_{i=0}^{\frac{2}{3}d}a_{i}x^{i}y^{\frac{2}{3}d-i}z^{\frac{d}{3}}$.
Since $f$ is harmonic,
\[
0=\Delta f=g\left(x,y\right)z^{\frac{d}{3}}+\sum_{i=0}^{\frac{2}{3}d}a_{i}\frac{d}{3}\left(\frac{d}{3}-1\right)x^{i}y^{\frac{2}{3}d-i}z^{\frac{d}{3}-2}
\]

where $g\left(x,y\right)$ is a polynomial in $x,y$.

We assume $f\neq0$ so $\exists i\,a_{i}\neq0$ and we have $\frac{d}{3}\left(\frac{d}{3}-1\right)=0\Longrightarrow d\in\left\{ 0,3\right\} $.

A straightforward calculation gives that the only harmonic homogeneous
polynomials $f$ of degree 3 such that the product $pf$ is harmonic
are linear combinations of $xyz$ and $x^{2}z-y^{2}z$, so the harmonic
functions $f$ such that $pf$ is harmonic are of the form $\alpha+\beta xyz+\gamma\left(x^{2}z-y^{2}z\right)$
with $\alpha,\beta,\gamma\in\mathbb{R}$.
\end{proof}
\begin{proof}[Proof of Theorem \ref{Theorem: pq for harmonic q and p with b}]
The proof is in the same spirit of the proof of Theorem~\ref{Theo:pq for q harmonic, p example case},
for details see \cite{adisMasters}.
\end{proof}

\section{The equivalence of Conjectures \ref{Conjecture:finite-share-zeros-b=00003D1}
and \ref{conj:(l,m)s.t.P2|Pl-finite}\label{sec:Equivalence-between-existence}}

In this section we explain why Conjectures \ref{Conjecture:finite-share-zeros-b=00003D1}
and \ref{conj:(l,m)s.t.P2|Pl-finite} are equivalent. We note that
this equivalence was already observed by Armitage \cite{Armitage-conesHarmonicVanish}.

Let $p\left(x,y,z\right)=x^{2}+y^{2}-2z^{2}$ and $Z=\left\{ p=0\right\} $.
It is known from \cite{malinnikova} that $h\in\mathcal{H}(Z)$ if
and only if there exists an analytic function $f$ such that $h=pf$.
In view of this, the equivalence of Conjectures \ref{Conjecture:finite-share-zeros-b=00003D1}
and \ref{conj:(l,m)s.t.P2|Pl-finite} follows from the following theorem.

\begin{thm}
\label{thm:equivalence:pq-harmonic,P2Q=00003DPlm}Let $l\in\mathbb{N}$.
The following statements are equivalent: 

\begin{enumerate}
\item Let $p\left(x,y,z\right)=x^{2}+y^{2}-2z^{2}$. There exists a harmonic
polynomial $h_{l}$ of degree $l$ such that $p|h_{l}$.
\item $\exists m\in\mathbb{N}$ such that $P_{2}\left(x\right)|P_{l}^{m}\left(x\right)$.
\end{enumerate}
\end{thm}
\begin{proof}
Assume statement 2. Notice that $p=r^{2}P_{2}\left(\cos\theta\right)$.
If $P_{2}|P_{l}^{m}$ let $h_{l}=r^{l}P_{l}^{m}\left(\cos\theta\right)e^{im\varphi}$.

Conversely, assume $h_{l}$ is a harmonic polynomial of degree $l$
such that $p|h_{l}$. Since $h_{l}$ is harmonic it can be written
in spherical coordinates in the form $\sum_{k=0}^{l}\sum_{m=-k}^{k}a_{k,m}r^{k}P_{k}^{m}\left(\cos\theta\right)e^{im\varphi}$
with some $a_{k,m}\in\mathbb{R}$. The polynomial $h_{l}$ vanishes
on $\theta_{\pm}$, where $\cos\left(\theta_{\pm}\right)=\pm\frac{1}{\sqrt{3}}$.
Since $\left\{ e^{im\varphi}\right\} $ and $\left\{ r^{k}\right\} $
are each a linearly independent set of functions, we get that $P_{l}^{m}\left(\cos\theta_{\pm}\right)=0$
for some~$m$.
\end{proof}

\section{\label{sec:Technical-Lemmas}Auxiliary Lemmas and Special Notation}
\begin{lem}
For non-negative integers $m,n$:\label{Lemma-integral-cosine} 
\[
\frac{1}{\pi}\int_{0}^{\pi}\cos^{n}\varphi\cos\left(m\varphi\right)d\varphi=\begin{cases}
\frac{1}{2^{n}}{n \choose \frac{n+m}{2}} & m\leq n\text{ and }m\equiv n\pmod{2}\\
0 & otherwise
\end{cases}
\]
\end{lem}
\begin{proof}
This becomes easy to check by writing $\cos\psi=\frac{e^{i\psi}+e^{-i\psi}}{2}$
and noting that $\Re\left(\frac{1}{\pi}\int_{0}^{\pi}e^{il\varphi}d\varphi\right)=\delta_{0,l}$. 
\end{proof}

\begin{proof}[Proof of Lemma \ref{lemma:forms of Hlm}]
We use formula (\ref{eq:plm-integral}) and set $\delta\in\left\{ 0,1\right\} $
such that $\left(l-m\right)\equiv\delta\pmod{2}$. 
\[
\frac{m!}{\left(l+m\right)!}P_{l}^{m}\left(x\right)=i^{m}\frac{1}{\pi}\int_{0}^{\pi}\left(x+y\cos\varphi\right)^{l}\cos\left(m\varphi\right)d\varphi
\]

\[
\underset{\text{Lemma \ref{Lemma-integral-cosine}}}{=}i^{m}\underset{j=0}{\sum^{\left\lfloor \frac{l-m}{2}\right\rfloor }}\begin{pmatrix}l\\
m+2j
\end{pmatrix}\frac{1}{2^{m+2j}}{m+2j \choose j}x^{l-m-2j}y^{m+2j}
\]
\[
\underset{(\ref{eq:x,y_by_z})}{=}x^{\delta}\left(iy\right)^{m}\underset{j=0}{\sum^{\left\lfloor \frac{l-m}{2}\right\rfloor }}\begin{pmatrix}l\\
m+2j
\end{pmatrix}\frac{1}{2^{m+2j}}{m+2j \choose j}\left(\frac{z}{z-4}\right)^{\left\lfloor \frac{l-m}{2}\right\rfloor -j}\left(\frac{4}{z-4}\right)^{j}
\]
\[
=\frac{x^{\delta}\left(iy\right)^{m}}{2^{m}\left(z-4\right)^{\left\lfloor \frac{l-m}{2}\right\rfloor }}\underset{j=0}{\sum^{\left\lfloor \frac{l-m}{2}\right\rfloor }}\begin{pmatrix}l\\
m+2j
\end{pmatrix}{m+2j \choose j}z^{\left\lfloor \frac{l-m}{2}\right\rfloor -j}
\]

\[
=\frac{x^{\delta}\left(iy\right)^{m}}{2^{m}\left(z-4\right)^{\left\lfloor \frac{l-m}{2}\right\rfloor }}\sum_{k=0}^{\left\lfloor \frac{l-m}{2}\right\rfloor }\begin{pmatrix}l\\
l-2k-\delta
\end{pmatrix}{l-2k-\delta \choose \left\lfloor \frac{l-m}{2}\right\rfloor -k}z^{k}
\]
\[
=\frac{x^{\delta}\left(iy\right)^{m}}{2^{m}\left(z-4\right)^{\left\lfloor \frac{l-m}{2}\right\rfloor }}\sum_{k=0}^{\left\lfloor \frac{l-m}{2}\right\rfloor }\frac{l!}{\left(2k+\delta\right)!\left(\left\lfloor \frac{l-m}{2}\right\rfloor -k\right)!\left(\left\lfloor \frac{l+m}{2}\right\rfloor -k\right)!}z^{k}
\]
\[
=\frac{x^{\delta}\left(iy\right)^{m}}{2^{m}\left(z-4\right)^{\left\lfloor \frac{l-m}{2}\right\rfloor }}\frac{l!}{\left(\left\lfloor \frac{l-m}{2}\right\rfloor \right)!\left(\left\lfloor \frac{l+m}{2}\right\rfloor \right)!}\sum_{k=0}^{\left\lfloor \frac{l-m}{2}\right\rfloor }\frac{\left(\left\lfloor \frac{l-m}{2}\right\rfloor \right)!\left(\left\lfloor \frac{l+m}{2}\right\rfloor \right)!}{\left(2k+\delta\right)!\left(\left\lfloor \frac{l-m}{2}\right\rfloor -k\right)!\left(\left\lfloor \frac{l+m}{2}\right\rfloor -k\right)!}z^{k}
\]
\[
=\frac{A_{l}^{m}x^{\delta}y^{m}}{\left(z-4\right)^{\left\lfloor \frac{l-m}{2}\right\rfloor }}\sum_{k=0}^{\left\lfloor \frac{l-m}{2}\right\rfloor }\frac{\left(-2\right)^{k}k!k!}{\left(2k+\delta\right)!}{\left\lfloor \frac{l-m}{2}\right\rfloor  \choose k}{\left\lfloor \frac{l+m}{2}\right\rfloor  \choose k}\left(-\frac{z}{2}\right)^{k}
\]
Recalling relation (\ref{eq:P-from-H}), the normalization $H_{l}^{m}\left(0\right)=1$
and Notation \ref{not:sigma-l-m}, we obtain the desired formulas.
\end{proof}
\begin{notation}
The dyadic valuation. \label{notation:The-dyadic-valuation.}
\end{notation}
For an integer $n$, denote

\[
v_{2}\left(n\right)=\begin{cases}
\max\left\{ \lambda\in\mathbb{N}\cup\left\{ 0\right\} \,:\,2^{\lambda}|n\right\}  & n\neq0\\
\infty & n=0
\end{cases}\,.
\]

For a rational number $\frac{m}{n}$ denote

\[
v_{2}\left(\frac{m}{n}\right)=v_{2}\left(m\right)-v_{2}\left(n\right)\,.
\]

\begin{lem}
$v_{2}\left(\sigma_{l}^{m}\left(k\right)\right)\geq v_{2}\left(k!\right)$\label{lem:power2 bigger than !}
\end{lem}
\begin{proof}
$\sigma_{l}^{m}\left(k\right)$ can also be written in the form
\[
\sigma_{l}^{m}\left(k\right)={x \choose y}\frac{z\left(z-1\right)...\left(z-k+1\right)}{1\cdot3\cdot5\cdot...\cdot\left(2k\pm1\right)}
\]
so there are $k$ consecutive numbers in the nominator and no even
numbers in the denominator.
\end{proof}

\begin{lem}
For $m\equiv2\pmod{4}$ and $l\equiv1\pmod{4}$, $\sigma_{l}^{m}\left(2\right)\equiv0\pmod{4}$.\label{lem:v2(sig(2))>2}
\end{lem}
\begin{proof}
Here $s$ is even and $t$ is odd. The only factors with powers of
two in this term are the following
\[
\frac{\left(s-t-1\right)\left(s+t-1\right)}{2}\equiv0\pmod{4}\,.
\]

\end{proof}
\bibliographystyle{plain}
\bibliography{harmonic_functions_vanishing_on_a_cone}

\end{document}